\documentclass[11pt, letterpaper]{amsart}
\usepackage[margin=1.5in]{geometry}

\usepackage[english]{babel}
\usepackage[utf8]{inputenc}
\usepackage[T1]{fontenc}
\usepackage{lmodern} 

\usepackage[colorinlistoftodos]{todonotes}

\usepackage{amsfonts} 
\usepackage{amssymb}
\usepackage{amsmath}
\usepackage{amsthm}
\usepackage{graphicx}
\usepackage{xifthen}
\usepackage{verbatim}
\usepackage{enumerate}
\usepackage{needspace}
\usepackage{mathtools}
\usepackage{xfrac}
\usepackage{tikz-cd}
\usepackage{lipsum}

\usepackage[colorlinks=true, allcolors=blue]{hyperref}

\usepackage[nameinlink]{cleveref} 

\DeclareRobustCommand{\SkipTocEntry}[5]{}

\newtheorem{theorem}{Theorem}[section]
\newtheorem{proposition}[theorem]{Proposition}
\newtheorem{corollary}[theorem]{Corollary}
\newtheorem{lemma}[theorem]{Lemma}

\theoremstyle{definition}
\newtheorem{definition}[theorem]{Definition}

\theoremstyle{definition}
\newtheorem{example}[theorem]{Example}

\theoremstyle{definition}
\newtheorem{remark}[theorem]{Remark}

\crefname{chapter}{chapter}{chapters}
\Crefname{chapter}{Chapter}{Chapters}
\crefname{section}{section}{sections}
\Crefname{section}{Section}{Sections}
\crefname{subsection}{section}{sections}
\Crefname{subsection}{Section}{Sections}
\crefname{subsubsection}{section}{sections}
\Crefname{subsubsection}{Section}{Sections}
\crefname{figure}{figure}{figures}
\Crefname{figure}{Figure}{Figures}
\crefname{table}{table}{tables}
\Crefname{table}{Table}{Tables}

\crefname{theorem}{theorem}{theorems}
\Crefname{theorem}{Theorem}{Theorems}
\crefname{proposition}{proposition}{propositions}
\Crefname{proposition}{Proposition}{Propositions}
\crefname{corollary}{corollary}{corollaries}
\Crefname{corollary}{Corollary}{Corollaries}
\crefname{lemma}{lemma}{lemmas}
\Crefname{lemma}{Lemma}{Lemmas}
\crefname{definition}{definition}{definitions}
\Crefname{definition}{Definition}{Definitions}
\crefname{conjecture}{conjecture}{conjectures}
\Crefname{conjecture}{Conjecture}{Conjectures}
\crefname{example}{example}{examples}
\Crefname{example}{Example}{Examples}
\crefname{remark}{remark}{remarks}
\Crefname{remark}{Remark}{Remarks}

\newcommand{\Z}{\mathbb{Z}}
\newcommand{\N}{\mathbb{N}}
\newcommand{\Q}{\mathbb{Q}}
\newcommand{\C}{\mathbb{C}}

\newcommand{\norm}[1]{ \left\| #1 \right\| }

\newcommand{\card}[1]{\left| #1 \right|}

\newcommand{\closure}[1]{\overline{#1}}

\newcommand{\set}[1]{\left\{#1\right\}}

\newcommand{\directsum}{\oplus}
\newcommand{\isoto}{\cong}
\newcommand{\boundary}{\partial}
\newcommand{\defeq}{\coloneqq}

\newcommand{\wtilde}[1]{\widetilde{#1}}

\newcommand{\actson}{\curvearrowright}

\newcommand{\tensor}{\otimes}
\newcommand{\sdprod}{\rtimes}

\DeclareMathOperator{\conv}{conv}

\DeclareMathOperator{\Sub}{Sub}
\DeclareMathOperator{\Fix}{Fix}
\newcommand{\normal}{\triangleleft}
\newcommand{\F}{\mathbb{F}}

\title{Relative C*-simplicity and characterizations for normal subgroups}
\author{Dan Ursu}
\address{Department of Pure Mathematics\\University of Waterloo\\
	200 University Avenue West, Waterloo, Ontario, N2L 3G1, Canada}
\email{dursu@uwaterloo.ca}
\thanks{This work was supported by the Natural Sciences and Engineering Research Council of Canada (NSERC) [grant number PGSD3-535032-2019]. Cette recherche a \'{e}t\'{e} financ\'{e}e par le Conseil de recherches en sciences naturelles et en g\'{e}nie du Canada (CRSNG) [num\'{e}ro de subvention PGSD3-535032-2019].}

\begin{document}
		
	\begin{abstract}
		The notion of a plump subgroup was recently introduced by Amrutam. This is a relativized version of Powers' averaging property, and it is known that Powers' averaging property is equivalent to C*-simplicity. With this in mind, we introduce a relativized notion of C*-simplicity, and show that for normal subgroups it is equivalent to plumpness, along with several other characterizations.
	\end{abstract}

	\maketitle

	\tableofcontents

	\section{Introduction and statement of main results}
		
	Throughout this paper, unless specified otherwise, $G$ denotes a discrete group, $H$ a subgroup of $G$, $N$ a normal subgroup of $G$, and $A$ a C*-algebra equipped with an action of $G$ by *-automorphisms. The reduced group C*-algebra of $G$ is denoted by $C^*_\lambda(G)$, the canonical trace on $C^*_\lambda(G)$ by $\tau_\lambda$, and the reduced crossed product of $A$ and $G$ by $A \rtimes_\lambda G$. All topological $G$-spaces will be assumed to be compact and Hausdorff.
	
	A recent result of Amrutam \cite[Theorem~1.1]{amrutam_plump_subgroups} gives a sufficient condition for all intermediate C*-subalgebras $B$ satisfying $C^*_\lambda(G) \subseteq B \subseteq A \rtimes_\lambda G$ to be of the form $A_1 \rtimes_\lambda G$ for some $G$-C*-subalgebra $A_1 \subseteq A$. Namely, he introduces the notion of a \textit{plump subgroup}, and proves that the above intermediate subalgebra property holds if $G$ has the approximation property (AP), and the kernel of the action $G \actson A$ contains a plump subgroup of $G$. For convenience, we recall the definition here:
	
	\begin{definition}
	\label{definition_plump_subgroup}
		A subgroup $H \leq G$ is \textit{plump} if for any $\varepsilon > 0$ and any finite $F \subseteq G \setminus \set{e}$, there are $s_1, \dots, s_m \in H$ such that
		$$ \norm{\frac{1}{m} \sum_{j=1}^m \lambda_{s_j} \lambda_t \lambda_{s_j}^*} < \varepsilon \quad \forall t \in F. $$
	\end{definition}

	However, the following remark shows that, for \cite[Theorem~1.1]{amrutam_plump_subgroups}, it suffices to consider only normal subgroups:
	
	\begin{remark}
		Assume $H \leq K \leq G$, and $H$ is plump in $G$. Then it is clear that $K$ is also plump in $G$. In particular, the kernel of the action $G \actson A$ contains a plump subgroup of $G$ if and only if the kernel itself is plump.
	\end{remark}

	Sufficient characterizations of plumpness are given in \cite[Section~3]{amrutam_plump_subgroups}. Recall that, if $N$ is a normal subgroup of $G$, then the action of $N$ on its Furstenberg boundary $\boundary_F N$ extends uniquely to an action of $G$ - see the review given in \Cref{subsection_preliminaries_boundary_theory}. It is shown that if $N$ is C*-simple and has trivial centralizer in $G$, then $G$ acts freely on $\boundary_F N$, which in turn implies $N$ is plump in $G$ \cite[Corollary~3.2]{amrutam_plump_subgroups}. One of the results we will show is that the converses to these statements also hold:
	
	\begin{theorem}
	\label{normal_plump_subgroup_amrutam_equivalence}
		Assume $N \normal G$ is normal. The following are equivalent:
		\begin{enumerate}
			\item $N$ is plump in $G$.
			\item The action $G \actson \boundary_F N$ is free.
			\item There exists some $G$-minimal, $N$-strongly proximal, $G$-topologically free space.
			\item $N$ is C*-simple and $C_G(N) = \set{e}$.
			\item $G$ is C*-simple and $C_G(N) = \set{e}$.
		\end{enumerate}
	\end{theorem}

	Setting $N = G$ in the above theorem gives back various equivalences between C*-simplicity and other characterizations. For a review of these characterizations, together with necessary definitions, see the review in \Cref{subsection_preliminaries_c_simplicity}.
	
	\begin{remark}
	\label{remark_plump_implies_c_simple}
		Plumpness is a relativized version of Powers' averaging property, and so setting $N = G$ in \Cref{normal_plump_subgroup_amrutam_equivalence}, we get back that $G$ is C*-simple if and only if it satisfies Powers' averaging property. In fact, if $G$ contains any (not necessarily normal) plump subgroup $H$, then we see that both $H$ and $G$ satisfy Powers' averaging property, and so both are C*-simple. Similarly, one also obtains the various dynamical characterizations of C*-simplicity by setting $N = G$.
	\end{remark}

	From here, it is natural to ask if plumpness is equivalent to some generalized notion of C*-simplicity. To answer this question, we introduce the notion of relative simplicity of C*-algebras, and using this, relative C*-simplicity for groups.
	
	\begin{definition}
		\label{definition_relative_c_simple}
		Assume $A$ is a unital C*-algebra, and $B \subseteq A$ is a unital sub-C*-algebra. We say that $B$ is \textit{relatively simple} in $A$ if any unital completely positive map $\phi : A \to B(\mathcal{H})$ which is a *-homomorphism on $B$ is faithful on $A$. Given $H \leq G$, we say that $H$ is \textit{relatively C*-simple} in $G$ if $C^*_\lambda(H)$ is relatively simple in $C^*_\lambda(G)$.
	\end{definition}

	\begin{theorem}
	\label{normal_plump_subgroup_relative_simplicity_equivalence}
		Assume $N \normal G$ is normal. The following are equivalent:
		\begin{enumerate}
			\item $N$ is plump in $G$.
			\item $N$ is relatively C*-simple in $G$.
			\item $C^*_\lambda(N)$ is relatively simple in $C(\boundary_F N) \rtimes_\lambda G$.
			\item $C(\boundary_F N) \rtimes_\lambda N$ is relatively simple in $C(\boundary_F N) \rtimes_\lambda G$.			
		\end{enumerate}
	\end{theorem}

	\begin{remark}
		For consistency, we will use the term \textit{relatively C*-simple} in place of \textit{plump} throughout the rest of this paper when it comes to normal subgroups.
	\end{remark}

	We may also ask what other characterizations of C*-simplicity generalize to an equivalent characterization of relative C*-simplicity. Kennedy's intrinsic characterization is one such result. For a review of this, along with a review of the Chabauty topology on the space of subgroups $\Sub(G)$, again see the review in \Cref{subsection_preliminaries_c_simplicity}.
	
	\begin{definition}
	\label{definition_relative_residual_normal_and_urs}
		Assume $H \leq G$. An \textit{$H$-uniformly recurrent subgroup} of $G$ is a (non-empty) closed $H$-minimal subset of $\Sub(G)$. It is called amenable if one (equivalently all) of its elements are amenable. It is called nontrivial if it is not $\set{\set{e}}$. A subgroup $K \leq G$ is called \text{$H$-residually normal} if the closed $H$-orbit of $K$ in $\Sub(G)$ does not contain the trivial subgroup $\set{e}$. Algebraically, $K \leq G$ is $H$-residually normal if and only if there exists a finite $F \subseteq G \setminus \set{e}$ such that $F \cap sKs^{-1} \neq \emptyset$ for any $s \in H$.
	\end{definition}

	\begin{theorem}
	\label{normal_plump_subgroup_residually_normal_equivalence}
		Assume $N \normal G$ is normal. The following are equivalent:
		\begin{enumerate}
			\item $N$ is relatively C*-simple in $G$.
			\item There is no amenable $N$-residually normal subgroup of $G$.
			\item There is no nontrivial amenable $N$-uniformly recurrent subgroup of $G$.
		\end{enumerate}
	\end{theorem}

	\addtocontents{toc}{\SkipTocEntry}
	\section*{Acknowledgements}
	
	The author would like to thank his supervisor, Matthew Kennedy, for giving detailed comments and suggestions throughout the development of this paper. In addition, the author would also like to thank Tattwamasi Amrutam, Mehrdad Kalantar, and Sven Raum for looking through a draft of this paper and giving helpful feedback.

	\section{Preliminaries}
	\label{section_preliminaries}
	
	\subsection{Boundary theory}
	\label{subsection_preliminaries_boundary_theory}
	
	Boundary theory was originally developed by Furstenberg in \cite{furstenberg_boundary_theory},
	and played an important role in \cite{kalantar_kennedy_boundaries} and \cite{breuillard_kalantar_kennedy_ozawa_c_simplicity}, which study C*-simplicity of discrete groups. For convenience, we recall all of the basic facts that we will use here. To establish notation, $G$ will always denote a discrete group.
	
	\begin{definition}
		Let $X$ be a compact Hausdorff space, and assume that $G$ acts by homeomorphisms on $X$. The action is \textit{minimal} if $X$ has no nontrivial closed $G$-invariant subsets. The action is \textit{strongly proximal} if for any Borel Radon probability measure $\mu \in P(X)$, the weak*-closed $G$-orbit of $\mu$ contains a Dirac mass $\delta_x$. A \textit{boundary} is a minimal and strongly proximal compact Hausdorff space.
	\end{definition}

	The appropriate notion of morphism between boundaries is a $G$-equivariant, continuous map.
	
	\begin{proposition}
		Morphisms between boundaries are unique, assuming they exist.
	\end{proposition}

	\begin{proof}
		This follows almost immediately from \cite[Proposition~4.2]{furstenberg_boundary_theory}.
	\end{proof}

	\begin{proposition}
		There is a universal boundary $\boundary_F G$, in the sense that every other boundary is the image of $\boundary_F G$ under some morphism. This universal boundary is also unique up to isomorphism.
	\end{proposition}

	The universal boundary $\boundary_F G$ given above is nowadays called the \textit{Furstenberg boundary}, and a proof of its existence can be found in \cite[Proposition~4.6]{furstenberg_boundary_theory}.
	
	Recall that an \textit{extremally disconnected} space is one where the closure of any open set is open. The following is a well-known theorem of Frol\'{i}k, and can be found in \cite[Theorem~3.1]{frolik_extremally_disconnected_fixed_points}.
	
	\begin{theorem}
	\label{frolik_theorem}
		The fixed point set of any homeomorphism of an extremally disconnected space is clopen.
	\end{theorem}
	
	\begin{corollary}
	\label{frolik_theorem_furstenberg_boundary}
		The fixed point set of any homeomorphism of $\boundary_F G$ is clopen.
	\end{corollary}
	
	\begin{proof}
		It is known that the Furstenberg boundary of a discrete group is always extremally disconnected - see \cite[Remark~3.16]{kalantar_kennedy_boundaries} or \cite[Proposition~2.4]{breuillard_kalantar_kennedy_ozawa_c_simplicity}.
	\end{proof}

	A well-known relativization fact that will come in extremely useful is the following:
	
	\begin{proposition}
		Assume $N \normal G$ is normal. The action of $N$ on $\boundary_F N$ extends uniquely to an action of $G$.
	\end{proposition}

	A proof of this fact can be found in \cite[Lemma~20]{ozawa_c_simplicity}. Note that, by uniqueness, there is no ambiguity when referring to \textit{the} action of $G$ on $\boundary_F N$.

	\subsection{C*-simplicity}
	\label{subsection_preliminaries_c_simplicity}
	
	Again, $G$ will always denote a discrete group. The group $G$ is called C*-simple if its reduced group C*-algebra $C^*_\lambda(G)$ is simple. Here, we collect the various characterizations of C*-simplicity that we will make use of throughout this paper.
	
	The Furstenberg boundary $\boundary_F G$ (see the review in \Cref{subsection_preliminaries_boundary_theory}) played a central role in the original characterizations of C*-simplicity. Recall that an action $G \actson X$ is said to be \textit{free} if the fixed point sets $\Fix(t)$ are empty for $t \neq e$. If $X$ is a topological space, a weaker notion is \textit{topologically free}, where the fixed point sets $\Fix(t)$ have empty interior for $t \neq e$.
	The following theorem is collectively proven in \cite[Theorem~3.1]{breuillard_kalantar_kennedy_ozawa_c_simplicity} (Theorem~\ref{c_simplicity_equivalence_dynamical}, $(i) \iff (iii) \iff (iv)$) and \cite[Theorem~6.2]{kalantar_kennedy_boundaries} (Theorem~\ref{c_simplicity_equivalence_dynamical}, $(i) \iff (ii)$, along with other equivalences).
	
	
	\begin{theorem}
	\label{c_simplicity_equivalence_dynamical}
		The following are equivalent:
		\begin{enumerate}
			\item $G$ is C*-simple.
			\item $C(\boundary_F G) \rtimes_\lambda G$ is simple.
			\item The action of $G$ on $\boundary_F G$ is free.
			\item The action of $G$ on some boundary is topologically free.
		\end{enumerate}
	\end{theorem}

	It is now known that C*-simplicity is equivalent to an averaging property originally considered by Powers. The definition we present here is easily seen to be equivalent to the definition presented in \cite[Definition~6.2]{kennedy_c_simplicity_intrinsic}.

	\begin{definition}
		A discrete group $G$ is said to satisfy \textit{Powers' averaging property} if for any $\varepsilon > 0$ and any finite $F \subseteq G \setminus \set{e}$, there are $s_1, \dots, s_m \in G$ such that
		$$ \norm{\frac{1}{m} \sum_{j=1}^m \lambda_{s_j} \lambda_t \lambda_{s_j}^*} < \varepsilon \quad \forall t \in F. $$
	\end{definition}

	The equivalence between C*-simplicity and Powers' averaging property was independently proven in \cite[Theorem~6.3]{kennedy_c_simplicity_intrinsic} and \cite[Theorem~4.5]{haagerup_new_look_c_simplicity}.

	\begin{theorem}
	\label{c_simplicity_equivalence_powers_averaging}
		The following are equivalent:
		\begin{enumerate}
			\item $G$ is C*-simple.
			\item $G$ has Powers' averaging property.
		\end{enumerate}
	\end{theorem}

	Finally, there is an intrinsic characterization of C*-simplicity by Kennedy. Recall that we may view the space of subgroups of $G$, i.e. $\Sub(G)$, as a closed (hence compact) subset of $2^G$. The corresponding topology on $\Sub(G)$ is known as the \textit{Chabauty topology}, and the space of amenable subgroups $\Sub_a(G)$ is again a closed (hence compact) subset of this space. More information can be found in \cite[Section~4]{kennedy_c_simplicity_intrinsic}. The following definitions can be found at the start of \cite[Section~4]{kennedy_c_simplicity_intrinsic} and in \cite[Definition~5.1]{kennedy_c_simplicity_intrinsic}.
	
	\begin{definition}
		A \textit{uniformly recurrent subgroup} of $G$ is a nonempty closed minimal subset of $\Sub(G)$. It is called \textit{amenable} if one (equivalently all) of its elements are amenable, and \textit{nontrivial} if it is not $\set{\set{e}}$. A subgroup $K \leq G$ is called \textit{residually normal} if the closed orbit in $\Sub(G)$ does not contain the trivial subgroup $\set{e}$. Algebraically, $K \leq G$ is residually normal if and only if there exists a finite $F \subseteq G \setminus \set{e}$ such that $F \cap sKs^{-1} \neq \emptyset$ for any $s \in G$.
	\end{definition}

	The following equivalence is proven in \cite[Theorem~4.1]{kennedy_c_simplicity_intrinsic} and \cite[Theorem~5.3]{kennedy_c_simplicity_intrinsic}:

	\begin{theorem}
	\label{c_simplicity_equivalence_intrinsic}
		The following are equivalent:
		\begin{enumerate}
			\item $G$ is C*-simple.
			\item $G$ has no nontrivial amenable uniformly recurrent subgroup.
			\item $G$ has no amenable residually normal subgroup.
		\end{enumerate}
	\end{theorem}

	\begin{remark}
	\label{remark_c_simple_countability}
		It is worth noting that countability of $G$ is not a requirement for any of the characterizations of C*-simplicity listed here. We will use some of these equivalences to prove our results, which in turn will be used for some examples, some of which involve uncountable groups.
	\end{remark}

	\section{Proof of main results}
	\label{section_proof_of_most_of_main_result}
	
	We first prove \Cref{normal_plump_subgroup_amrutam_equivalence}, most of which is already proven in \cite{amrutam_plump_subgroups}. First, we dualize the definition of plumpness to pass from the C*-algebra $C^*_\lambda(G)$ to its state space:
	
	\begin{lemma}
	\label{lemma_plump_dualization}
		Assume $H \leq G$. Then $H$ is plump in $G$ if and only if for any $\phi \in S(C^*_\lambda(G))$, the closed convex hull $\closure{\conv} H\phi$ contains the canonical trace $\tau_\lambda$.
	\end{lemma}

	\begin{proof}
		The proof is analogous to the proof given in \cite[Theorem~4.5]{haagerup_new_look_c_simplicity}.
	\end{proof}

	\begin{lemma}
		\label{lemma_G_minimal_N_strongly_proximal_implies_N_minimal}
		Assume $N \normal G$ is normal. Then any $G$-minimal, $N$-strongly proximal space $X$ is also $N$-minimal.
	\end{lemma}
	
	\begin{proof}
		Given that $N$-minimal components are always disjoint, it follows from strong proximality that there can only be exactly one $N$-minimal component in $X$ - call it $M$. Further, we note that any translate $tM$ (where $t \in G$) is still $N$-invariant. Indeed, $NtM = tNM = tM$, and so $M \subseteq tM$ by uniqueness. Using this, we also obtain $tM \subseteq t(t^{-1}M) = M$, and so $M$ is $G$-invariant. But $X$ is assumed to be $G$-minimal, and so $M = X$, i.e. $X$ is $N$-minimal.
	\end{proof}

	\begin{lemma}
	\label{lemma_topologically_free_boundary_implies_free_furstenberg_boundary}
		Assume that $N$ and $X$ are as in \Cref{lemma_G_minimal_N_strongly_proximal_implies_N_minimal}, and $\phi : \boundary_F N \to X$ is a $G$-equivariant continuous map. If, in addition, $X$ is $G$-topologically free, then the action of $G$ on $\boundary_F N$ is free.
	\end{lemma}

	\begin{proof}
		Assume otherwise, so that there is some $t \in G$ with $U \defeq \Fix_{\boundary_F N}(t)$ nonempty. It is known that $U$ is necessarily clopen (see \Cref{frolik_theorem_furstenberg_boundary}). We know that $\boundary_F N = s_1 U \cup \dots \cup s_n U$ for some $s_i \in N$, by minimality and compactness. Thus, $X = s_1 \phi(U) \cup \dots \cup s_n \phi(U)$, and so $\phi(U)$ being closed implies it has non-empty interior. But $\phi(U) \subseteq \Fix_X(t)$, contradicting topological freeness. This shows that the action of $G$ on $\boundary_F N$ is free.
	\end{proof}

	\begin{remark}
		This argument is analogous to the proof of \cite[Lemma~3.2]{breuillard_kalantar_kennedy_ozawa_c_simplicity}, which claims that maps $\pi : Y \to X$ between minimal $G$-spaces map closed sets of nonempty interior to closed sets of nonempty interior. However, without the additional assumption that $U$ is clopen (hence $\phi(U)$ is closed), one cannot conclude that $\phi(U)$ has interior just from having $X = s_1 \phi(U) \cup \dots \cup s_n \phi(U)$. Consider, for example, $[0,1] = ([0,1] \cap \Q) \cup ([0,1] \cap \Q^{\complement})$. Hence, the above lemma also serves as a slight correction to the proof of \cite[Theorem~3.1, $(2) \iff (3)$]{breuillard_kalantar_kennedy_ozawa_c_simplicity}, which uses \cite[Lemma~3.2]{breuillard_kalantar_kennedy_ozawa_c_simplicity} as a prerequisite.
	\end{remark}

	\begin{remark}
	\label{remark_sven_proof}
		Before beginning the proof of \Cref{normal_plump_subgroup_amrutam_equivalence}, the author would like to thank Sven Raum for giving a much cleaner proof of $(3) \implies (2)$, which is the argument presented here. The original proof can be found in \Cref{section_universal_G_minimal_H_strongly_proximal}.
	\end{remark}

	\begin{proof}[Proof of \Cref{normal_plump_subgroup_amrutam_equivalence}]
		The implications $(4) \implies (2) \implies (1)$ are given in \cite[Theorem~3.1, Corollary~3.2]{amrutam_plump_subgroups}. Further, $(4) \iff (5)$ follows easily from \cite[Theorem~1.4]{breuillard_kalantar_kennedy_ozawa_c_simplicity}. It is clear that $(2) \implies (3)$ holds, as $\boundary_F N$ is such a space.
		
		To show that $(3) \implies (2)$ holds, assume $X$ is such a space. By \Cref{lemma_G_minimal_N_strongly_proximal_implies_N_minimal}, we have that $X$ is in fact $N$-minimal, and therefore an $N$-boundary. Hence, we obtain an $N$-equivariant continuous map $\phi : \boundary_F N \to X$. We claim it is $G$-equivariant. Letting $s \in N$ and $t \in G$, and slightly abusing notation by directly viewing these as automorphisms on $\boundary_F N$ and $X$, we have that
		$$ (t \circ \phi \circ t^{-1})(sy) = (t \circ \phi)(t^{-1}st t^{-1}y) = t t^{-1}st (\phi(t^{-1}y)) = s ((t \circ \phi \circ t^{-1})(y)). $$
		As morphisms between boundaries (in this case, $N$-equivariant continuous maps) are necessarily unique, we have that $t \circ \phi \circ t^{-1} = \phi$, or in other words, $\phi$ is $G$-equivariant. By \Cref{lemma_topologically_free_boundary_implies_free_furstenberg_boundary}, we have that the action of $G$ on $\boundary_F N$ is free.
		
		It remains to show $(1) \implies (4)$. To this end, we note that $N$ is C*-simple by \Cref{remark_plump_implies_c_simple}. Assume it is not the case that $C_G(N) = \set{e}$, and choose a nontrivial amenable subgroup $K \leq C_G(N)$. We know that the canonical character $1_K : K \to \C$ extends to a *-homomorphism $1_K : C^*_\lambda(K) \to \C$, and that there is also a canonical conditional expectation $E_K : C^*_\lambda(G) \to C^*_\lambda(K)$ mapping $\lambda_t$ to itself if $t \in K$, and zero otherwise. It is easy to check that the composition $1_K \circ E_K : C^*_\lambda(G) \to \C$ is an $N$-fixed state, which contradicts \Cref{lemma_plump_dualization}.
	\end{proof}

	We now aim to prove \Cref{normal_plump_subgroup_relative_simplicity_equivalence}. Some easy observations about relative simplicity as defined in \Cref{definition_relative_c_simple} are in place.
	
	\begin{proposition}
	\label{relative_simplicity_containment}
		Let $A$, $B$, and $C$ denote unital C*-algebras.
		\begin{enumerate}
			\item $A$ is relatively simple in itself if and only if it is simple.
			\item If $A \subseteq B \subseteq C$ with $A$ relatively simple in $C$, then $B$ is simple.
			\item If $A \subseteq B \subseteq C$ with $A$ relatively simple in $C$, then $A$ is relatively simple in $B$ and $B$ is relatively simple in $C$.
		\end{enumerate}
	\end{proposition}

	\begin{proof}
		First, to prove $(3)$, let $\phi : B \to B(\mathcal{H})$ be a unital completely positive map which is a *-homomorphism on $A$. This extends to a unital completely positive map $\wtilde{\phi} : C \to B(\mathcal{H})$, which is faithful by assumption, and so $\phi$ is faithful, showing $A$ is relatively simple in $B$. It is clear that $B$ is relatively simple in $C$ almost by definition. Claim $(1)$ follows from the fact that for *-homomorphisms, faithfulness and injectivity are equivalent. Finally, $(2)$ follows by applying $(3)$ to get that $A$ is relatively simple in $B$, then applying $(3)$ again to the containment $A \subseteq B \subseteq B$ to conclude that $B$ is relatively simple in itself, and finally applying $(1)$.
	\end{proof}

	We also make use of the following lemma:
	
	\begin{lemma}
	\label{lemma_plump_pseudoexpectation_equivalence}
		Assume $H \leq G$. Then $H$ is plump in $G$ if and only if the only $H$-equivariant unital completely positive map $\phi : C^*_\lambda(G) \to C(\boundary_F H)$ is the canonical trace $\tau_\lambda$.
	\end{lemma}

	\begin{proof}
		Observe that \Cref{lemma_plump_dualization} tells us that, because $\tau_\lambda$ is $H$-invariant, $H$ is plump if and only if there are no nontrivial $H$-irreducible closed convex subsets of $S(C^*_\lambda(G))$. As the closure of the extreme points of any such subset is an $H$-boundary \cite[Theorem~III.2.3]{glasner_proximal_flows}, this is equivalent to there being no nontrivial $H$-boundaries in $S(C^*_\lambda(G))$. From here, the proof is analogous to that of \cite[Proposition~3.1]{kennedy_c_simplicity_intrinsic}.
	\end{proof}

	\begin{proof}[Proof of \Cref{normal_plump_subgroup_relative_simplicity_equivalence}]
		$(1) \implies (3)$ This argument is adapted from part of the proof of $(2) \implies (1)$ in \cite[Theorem~3.1]{breuillard_kalantar_kennedy_ozawa_c_simplicity}. Assume $\phi : C(\boundary_F N) \rtimes_\lambda G \to B(\mathcal{H})$ is unital and completely positive, and also a *-homomorphism on $C^*_\lambda(N)$. We may equip $B(\mathcal{H})$ with an $N$-action given by $s \cdot T = \phi(\lambda_s) T \phi(\lambda_s)^*$ for $s \in N$ and $T \in B(\mathcal{H})$. Further, $\phi$ is $N$-equivariant with respect to this action on $B(\mathcal{H})$, as $C^*_\lambda(N)$ lies in the multiplicative domain of $\phi$. We also have that, by injectivity, there is an $N$-equivariant unital completely positive map $\psi : B(\mathcal{H}) \to C(\boundary_F N)$. We note that $\psi \circ \phi : C(\partial_F N) \rtimes_\lambda G \to C(\boundary_F N)$ restricts to the canonical trace on $C^*_\lambda(G)$ by \Cref{lemma_plump_pseudoexpectation_equivalence}. Furthermore, $\psi \circ \phi$ is the identity on $C(\boundary_F N)$ by rigidity, and so $C(\boundary_F N)$ lies in the multiplicative domain of this map. Combining these two observations yields that $\psi \circ \phi$ is the canonical expectation, which is faithful, and so $\phi$ is faithful.
		
		$(3) \implies (2)$ This follows from \Cref{relative_simplicity_containment}.
		
		$(2) \implies (1)$ Assume $(1)$ does not hold. We know that $N$ must be C*-simple by \Cref{relative_simplicity_containment}, and so looking back at \Cref{normal_plump_subgroup_amrutam_equivalence}, it must be the case that $C_G(N) \neq \set{e}$. Choose any nontrivial amenable subgroup $K$ of $C_G(N)$, and note that $N \cap K = \set{e}$, as $N$ has trivial center (C*-simplicity implies trivial amenable radical). Thus, $NK \isoto N \times K$, and so $C^*_\lambda(NK) \isoto C^*_\lambda(N) \tensor C^*_\lambda(K)$. Letting $\lambda_N : C^*_\lambda(N) \to B(\ell^2(N))$ denote the extension of the left-regular representation of $N$ to $C^*_\lambda(N)$, and $1_K : C^*_\lambda(K) \to \C$ the extension of the trivial character, we have that $\lambda_N \tensor 1_K : C^*_\lambda(N) \tensor C^*_\lambda(K) \to B(\ell^2(N))$ is a non-faithful *-homomorphism. Thus, any unital completely positive extension $\phi : C^*_\lambda(G) \to B(\ell^2(N))$ will be non-faithful, yet be a *-homomorphism on $C^*_\lambda(N)$.
		
		$(3) \implies (4)$ This follows from \Cref{relative_simplicity_containment}.
		
		$(4) \implies (1)$ This implication will be quite similar to $(2) \implies (1)$. Assume that $(1)$ doesn't hold, and observe that by \Cref{relative_simplicity_containment}, $C(\boundary_F N) \rtimes_\lambda N$ is simple, which is known to imply that $N$ is C*-simple (see \Cref{c_simplicity_equivalence_dynamical}). Again, $C_G(N) \neq \set{e}$, and choosing any nontrivial amenable subgroup $K \leq C_G(N)$, we have $NK \isoto N \times K$. Further, $K$ acts trivially on $\boundary_F N$ by \cite[Lemma~5.3]{breuillard_kalantar_kennedy_ozawa_c_simplicity}, and so $C(\boundary_F N) \rtimes_\lambda (NK) \isoto (C(\boundary_F N) \rtimes_\lambda N) \tensor C^*_\lambda(K)$. Letting $\pi : C(\boundary_F N) \rtimes_\lambda N \to B(\mathcal{H})$ be any (necessarily faithful) representation, and $1_K : C^*_\lambda(K) \to \C$ be the extension of the trivial character, we have that $\pi \tensor 1_K : (C(\boundary_F N) \rtimes_\lambda N) \tensor C^*_\lambda(K) \to B(\mathcal{H})$ is a non-faithful *-homomorphism. Any unital completely positive extension to $C^*_\lambda(G)$ will contradict the assumption of $(4)$.
	\end{proof}

	\begin{proof}[Proof of \Cref{normal_plump_subgroup_residually_normal_equivalence}]
		$(2) \implies (1)$ Assume $(1)$ does not hold. Applying \Cref{normal_plump_subgroup_amrutam_equivalence}, we either have that $N$ is not C*-simple, or $C_G(N) \neq \set{e}$. If $N$ is not C*-simple, then by Kennedy's intrinsic characterization of C*-simplicity (see \Cref{c_simplicity_equivalence_intrinsic}), $N$ has a nontrivial amenable $N$-residually normal subgroup. If, on the other hand, $C_G(N) \neq \set{e}$, then any nontrivial amenable subgroup of $C_G(N)$ is $N$-residually normal.
		
		$(1) \implies (2)$ This is analogous to \cite[Remark~4.2]{kennedy_c_simplicity_intrinsic}. For convenience, we give the modified argument here. Assume $(2)$ doesn't hold, and $K$ is a nontrivial amenable $N$-residually normal subgroup of $G$. Amenability tells us that there is some $K$-invariant measure $\mu \in P(\boundary_F N)$. Strong proximality tells us that there is a net $(s_\lambda) \subseteq N$ with $s_\lambda \mu \to \delta_x$ for some $x \in \boundary_F N$. Dropping to a subnet, we may assume that $(s_\lambda K s_\lambda^{-1})$ is also convergent to some $L$, and $L \neq \set{e}$ by assumption. Chopping off the start of our net, we may in addition assume that there is some $l \in L \setminus \set{e}$ with $l \in s_\lambda K s_\lambda^{-1}$ for all $\lambda$, i.e. $l = s_\lambda k_\lambda s_\lambda^{-1}$ for some $k_\lambda \in K$. From here, we note that
		$$ l s_\lambda \mu \to l \delta_x = \delta_{lx}, $$
		while we also have
		$$ l s_\lambda \mu = s_\lambda k_\lambda \mu = s_\lambda \mu \to \delta_x. $$
		This shows $lx = x$, and so $G$ cannot act freely on $\boundary_F N$.
		
		$(2) \iff (3)$ If $K$ is any nontrivial amenable $N$-residually normal subgroup of $G$, then any $N$-minimal subset of the closed $N$-orbit of $K$ is an $N$-uniformly recurrent subgroup. Conversely, any element of an $N$-uniformly recurrent subgroup is $N$-residually normal by definition.
	\end{proof}
		
	We conclude this section with some remarks.

	\begin{remark}
	\label{remark_relative_c_simplicity_countability}
		Countability of $N$ or $G$ is not a requirement for any of the above proofs, nor is it required for any of the C*-simplicity analogues of the above characterizations obtained by setting $N = G$ (see \Cref{remark_c_simple_countability}), some of which were applied here.
	\end{remark}

	\begin{remark}
		Some of the characterizations we have given are closed under taking supergroups. Namely, if $H \leq G$ is any (not necessarily normal) subgroup that satisfies any of \Cref{normal_plump_subgroup_amrutam_equivalence} $(3)$ or $(5)$, \Cref{normal_plump_subgroup_relative_simplicity_equivalence} $(2)$, or \Cref{normal_plump_subgroup_residually_normal_equivalence} $(2)$ or $(3)$, then so does any subgroup of $G$ containing $H$. In particular, any normal subgroup of $G$ containing $H$ is relatively C*-simple.
	\end{remark}

	\section{The universal $G$-minimal, $H$-strongly proximal space}
	\label{section_universal_G_minimal_H_strongly_proximal}
	
	This section was originally dedicated to proving $(3) \implies (2)$ in \Cref{normal_plump_subgroup_amrutam_equivalence}, until a much cleaner proof was suggested by Sven Raum - see \Cref{remark_sven_proof}. The existence of a type of relative Furstenberg boundary with respect to arbitrary (not necessarily normal) subgroups might still be interesting, and for this reason this section is still kept.
	
	\begin{proposition}
	\label{universal_G_minimal_H_strongly_proximal_space_existence_uniqueness}
		Assume $H$ is a (not necessarily normal) subgroup of $G$. There is a universal $G$-minimal, $H$-strongly proximal $G$-space $B(G,H)$, in the sense that any other such space $X$ is a $G$-equivariant continuous image of $B(G,H)$. Further, this space is unique up to $G$-equivariant homeomorphism.
	\end{proposition}

	\begin{proof}
		The proof is quite similar to the topological proof of the existence of the Furstenberg boundary, a very brief sketch of which is given in \cite[Proposition~4.6]{furstenberg_boundary_theory}. We fill in the details and modify the argument appropriately here.
		
		Let $\set{X_\alpha}_{\alpha \in A}$ denote the set of all up-to-isomorphism $G$-minimal, $H$-strongly proximal spaces, where isomorphism refers to $G$-equivariant homeomorphism. Note that these can indeed be put into a set, as all of these spaces are necessarily the continuous image of $\beta G$ by minimality, so there is a limit on the cardinality of these spaces.
		
		We claim that the space $X \defeq \prod_\alpha X_\alpha$ is still $H$-strongly proximal. To see this, given any measure $\mu \in P(X)$, we note that for any $\alpha \in A$, there is a net $(h_\lambda) \subseteq H$ with $(h_\lambda \mu)$ converging to a Dirac mass when restricted to $C(X_\alpha)$. From here, it is easy to see that this can be done on finitely many $\alpha_1, \dots, \alpha_n \in A$. Now for each finite $F \subseteq A$, letting $\mu_F \in \closure{H\mu}$ be a Dirac mass when restricted to each $C(X_\alpha)$ for $\alpha \in F$, we note that any cluster point of the net $(\mu_F)$ (indexed over finite subsets of $A$, ordered under inclusion) is necessarily a Dirac mass on the entire space $X$.
		
		Let $B(G,H)$ be a $G$-minimal subset of $X$. It is clear that this space is still $H$-strongly proximal. We will also show that it is universal. Given any $X_\alpha$, consider the coordinate projection map $\pi_\alpha : X \to X_\alpha$. We see that $\pi_\alpha|_{B(G,H)} : B(G,H) \to X_\alpha$ is still surjective, as the image of this map is closed and $G$-invariant, and $X_\alpha$ is $G$-minimal.
		
		Finally, this space is unique up to isomorphism. Indeed, if $B'$ is another universal space, then there are $G$-equivariant continuous maps $\phi_1 : B(G,H) \to B'$ and $\phi_2 : B' \to B(G,H)$. But their compositions $\phi_2 \circ \phi_1 : B(G,H) \to B(G,H)$ and $\phi_1 \circ \phi_2 : B' \to B'$ are necessarily the respective identity maps between these spaces, as these spaces are both $G$-boundaries, and, assuming they exist, morphisms between $G$-boundaries are unique.
	\end{proof}

	\begin{remark}
		A different notion of relative Furstenberg boundary is presented in \cite{monod_relative_furstenberg_boundary}, and so we avoid using the term \textit{Furstenberg boundary} and notation $\boundary(G,H)$ to describe the universal object from \Cref{universal_G_minimal_H_strongly_proximal_space_existence_uniqueness}. Our notation $B(G,H)$ is derived from Furstenberg's notation $B(G)$ for the Furstenberg boundary of $G$, given in \cite[Proposition~4.6]{furstenberg_boundary_theory}.
	\end{remark}	

	\begin{corollary}
	\label{universal_G_minimal_N_strongly_proximal_is_N_furstenberg_boundary}
		If $N \normal G$ is normal, the universal $G$-minimal, $N$-strongly proximal space is $\boundary_F N$.
	\end{corollary}

	\begin{proof}
		Letting $B(G,N)$ denote the universal such space, there is a $G$-equivariant continuous surjection $\phi_1 : B(G,N) \to \boundary_F N$. However, \Cref{lemma_G_minimal_N_strongly_proximal_implies_N_minimal} tells us that $B(G,N)$ is in fact an $N$-boundary, and so there is an $N$-equivariant continuous surjection $\phi_2 : \boundary_F N \to B(G,N)$. The composition $\phi_1 \circ \phi_2 : \boundary_F N \to \boundary_F N$ is $N$-equivariant, and thus necessarily the identity map. This shows $\phi_2$ is injective, hence bijective. Thus, $\phi_1$ is also bijective, and therefore the isomorphism we are looking for.
	\end{proof}

	It is worth emphasizing a subtle point - \Cref{lemma_G_minimal_N_strongly_proximal_implies_N_minimal} tells us that any $G$-minimal, $N$-strongly proximal space is the $N$-equivariant image of $\boundary_F N$. However, \Cref{universal_G_minimal_N_strongly_proximal_is_N_furstenberg_boundary} gives us a $G$-equivariant map.
	
	\begin{proof}[Proof of \Cref{normal_plump_subgroup_amrutam_equivalence}, $(3) \implies (2)$]
		Assume $X$ is $G$-minimal, $N$-strongly proximal, and $G$-topologically free. By \Cref{universal_G_minimal_N_strongly_proximal_is_N_furstenberg_boundary}, there is a $G$-equivariant continuous map $\phi : \boundary_F N \to X$. By \Cref{lemma_topologically_free_boundary_implies_free_furstenberg_boundary}, the action of $G$ on $\boundary_F N$ is free.
	\end{proof}

	\section{Examples}
	
	It is worth noting that the characterization of being C*-simple and having trivial centralizer, originally given as a sufficient condition in \cite[Corollary~3.2]{amrutam_plump_subgroups}, is perhaps the ``nicest'' characterization of relative C*-simplicity. As such, some of the examples below will still be proven with this result, as opposed to our new results.
	
	\subsection{Free products}
	
	Given that the canonical example of a C*-simple group is $\F_2$, the free group on two generators, it is worthwhile to use this as an easy example. We will re-prove the following special case of \cite[Example~3.8]{amrutam_plump_subgroups} using one of our new results.
	
	\begin{example}
		Let $\F_2 = \left<a,b\right>$ denote the free group on two generators. Then the normal closure of $a$ is relatively C*-simple in $\F_2$.
	\end{example}

	\begin{proof}
		The Nielsen-Schreier theorem tells us that any subgroup of a free group is free. Thus, the only nontrivial amenable subgroups of $\F_2$ are the cyclic subgroups. Given any such subgroup $\left<x\right>$, assume first that the reduced word of $x$ starts with $b$ or $b^{-1}$. Then the reduced word length of $a^n x a^{-n}$ is eventually strictly increasing, showing that $\left<a^n x a^{-n}\right> \to \set{e}$ in the Chabauty topology. This is also true if the reduced word of $x$ ends with $b$ or $b^{-1}$. Finally, if both the start and end of $x$ lie in $\set{a,a^{-1}}$, then the reduced word length of $(bab^{-1})^n x (bab^{-1})^{-n}$ is strictly increasing, and so $\left<(bab^{-1})^n x (bab^{-1})^{-n}\right>$ is Chabauty-convergent to $\set{e}$. By \Cref{normal_plump_subgroup_residually_normal_equivalence}, we are done.
	\end{proof}

	We will also generalize \cite[Example~3.8]{amrutam_plump_subgroups} to free products as follows:
	
	\begin{theorem}
		Assume $G = H*K$, where $H$ and $K$ are nontrivial, and they are also not both $\Z_2$. Then any nontrivial normal subgroup of $G$ is relatively C*-simple.
	\end{theorem}

	\begin{proof}
		It is known that any such group is C*-simple \cite{paschke_salinas_free_products_c_simple}. Hence, any normal subgroup $N \normal G$ is C*-simple as well by \cite[Theorem~1.4]{breuillard_kalantar_kennedy_ozawa_c_simplicity}. It remains to show that any nontrivial normal subgroup has trivial centralizer. Assume otherwise, so that there exists some normal subgroup $N \neq \set{e}$ with $C_G(N) \neq \set{e}$, and pick nontrivial elements $x \in N$ and $y \in C_G(N)$. C*-simplicity of $N$ tells us that $N$ has trivial center, i.e. $N \cap C_G(N) = \set{e}$, and so $\left<x,y\right> \isoto \left<x\right> \times \left<y\right>$. But the Kurosh subgroup theorem tells us that
		$$ \left<x\right> \times \left<y\right> = \F_X * \prod_{i \in I}^* s_i H_i s_i^{-1} * \prod_{j \in J}^* t_j K_j t_j^{-1} $$
		for some subset $X \subseteq G$ and subgroups $H_i \leq H$, $K_j \leq K$. There are two cases when such a subgroup is abelian. The first case is if $X$ is a singleton, and $I$ and $J$ are empty. This is impossible, as $\left<x\right> \times \left<y\right>$ is not isomorphic to $\Z$. The second case is if, without loss of generality, $\left<x\right> \times \left<y\right>$ is some conjugate $s_i H_i s_i^{-1}$, where $H_i \leq H$. Equivalently, $\langle s_i^{-1} x s_i \rangle \times \langle s_i^{-1} y s_i \rangle$ is a subgroup of $H$. But both $N$ and $C_G(N)$ are normal subgroups of $G$, and so this says that there are some nontrivial $s \in N$ and $t \in C_G(N)$ that both lie in $H$. However, if we choose any nontrivial $r \in K$, then $t$ cannot commute with $rsr^{-1} \in N$, a contradiction.
	\end{proof}

	\subsection{Direct products}
	
	Taking direct sums and direct products of existing examples can provide some easy new examples:
	
	\begin{lemma}
	\label{lemma_direct_sum_product_of_c_simple_is_c_simple}
		Let $(G_i)$ be a family of C*-simple groups. Then both $\directsum G_i$ and $\prod G_i$ are C*-simple.
	\end{lemma}

	\begin{proof}
		We know that each $G_i$ acts freely on its Furstenberg boundary $\boundary_F G_i$ (see \Cref{c_simplicity_equivalence_dynamical}). From here, it is not hard to check that $\prod \boundary_F G_i$ is a free boundary action for both $\directsum G_i$ and $\prod G_i$, and so both of these groups are C*-simple by the same theorem.
	\end{proof}
	
	\begin{theorem}
		Let $(G_i)$ be a family of groups with relatively C*-simple normal subgroups $N_i \normal G_i$. The direct sum $\directsum N_i$ is relatively C*-simple in the direct product $\prod G_i$.
	\end{theorem}

	\begin{proof}
		Observe that $\directsum N_i$ is normal in $\prod G_i$, and the commutator of this subgroup is $\prod C_{G_i}(N_i)$. By \Cref{normal_plump_subgroup_amrutam_equivalence}, this commutator is trivial, and so applying this theorem again together with \Cref{lemma_direct_sum_product_of_c_simple_is_c_simple}, $\directsum N_i$ is relatively C*-simple in $\prod G_i$.
	\end{proof}

	\begin{remark}
		This shows that there exists an uncountable group with a countable relatively C*-simple normal subgroup, for example $\directsum_{n \in \N} \F_2 \normal \prod_{n \in \N} \F_2$. From the C*-algebras perspective, there is a non-separable C*-algebra with a separable relatively simple sub-C*-algebra.
	\end{remark}

	\subsection{Wreath products}
	
	Recall that the (unrestricted) wreath product $G \wr H$ is $(\prod_H G) \sdprod H$, where $H$ acts by left-translation on $\prod_H G$.
	
	\begin{theorem}
	\label{relative_simplicity_wreath_products}
		Assume $N$ is a relatively C*-simple subgroup of some group $G \neq \set{e}$, and let $H$ be any arbitrary group. Then the direct sum $\directsum_H N$ is relatively C*-simple in $G \wr H$.
	\end{theorem}

	\begin{proof}
		Note that $\directsum_H N$ is normal in $G \wr H$. It is easy to check that the canonical action of $\prod_H G$ on $\prod_H \boundary_F N$, together with the action of $H$ on $\prod_H \boundary_F N$ by left-translation, extend to an action of all of $G \wr H$. It is also not hard to see that $\directsum_H N$ acts strongly proximally and $G \wr H$ acts minimally on this space.
				
		It remains to show that the action of $G \wr H$ is still topologically free. To this end, first consider any nontrivial element of the form $((g_h),e) \in G \wr H$. Its fixed point set is given by $\prod_H \Fix(g_h)$, which is empty by \Cref{normal_plump_subgroup_amrutam_equivalence} and the assumption that at least one $g_h \neq e$. Now given any element $((g_h),h_0) \in G \wr H$ with $h_0 \neq e$, we note that
		$$ ((g_h),h_0)(x_h) = (g_h x_{h_0^{-1}h}), $$
		and so $(x_h)$ lies in the fixed point set of this element if and only if $g_h x_{h_0^{-1}h} = x_h$ for all $h$. In particular, setting $h = h_0$ gives us $g_{h_0} x_e = x_{h_0}$. If $\Fix((g_h),h_0)$ were to have interior, then it would contain a basic open subset of the form $\prod_H U_h$, where $U_h \subseteq X$ is open, and all but finitely many $U_h = X$. Given that $N$ is C*-simple and $N \neq \set{e}$ (as $N$ is relatively C*-simple in $G \neq \set{e}$ by assumption) we know that $\boundary_F N$ has no isolated points \cite[Proposition~3.15]{kalantar_kennedy_boundaries}, and so no $U_h$ is a singleton. But given that $x_{h_0}$ is entirely determined by the value $x_e$ takes, this cannot be the case. We conclude that $\Fix((g_h),h_0)$ has empty interior, and so by \Cref{normal_plump_subgroup_amrutam_equivalence}, we are done.
	\end{proof}

	\begin{remark}
		The sufficient condition for plumpness given in \cite[Lemma~3.5]{amrutam_plump_subgroups} assumes the group is countable and has countable fixed point sets. The proof of \Cref{relative_simplicity_wreath_products}, however, gives a natural class of topologically free boundary actions admitting uncountably many fixed points. Here, we see that $H \leq G \wr H$ fixes the diagonal of $\prod_H \boundary_F N$, and $\boundary_F N$ is uncountable as it is a nontrivial compact Hausdorff space with no isolated points. One could also replace $\boundary_F N$ by any $G$-minimal, $N$-strongly proximal, $G$-topologically free space $X$, and so any element $((g_h),e) \in G \wr H$ admits the fixed point set $\prod_H \Fix(g_h)$, which is uncountable if, for example, $\Fix(g_h)$ are nonempty for all $h$, and at least one $g_h = e$. Finally, while wreath products $G \wr H$ are often uncountable (for example, if $G \neq \set{e}$ and $H$ is infinite), the same observations hold for the restricted wreath product $(\directsum_H G) \sdprod H$ as well, which is countable if $G$ and $H$ are countable.
	\end{remark}

	\subsection{Groups with trivial center and only cyclic amenable subgroups}
	
	It is sometimes the case that the only amenable subgroups of a given group are the cyclic subgroups. For example, this property is true of the free groups by the Nielsen-Schreier theorem. Our aim is to show the following:
	
	\begin{theorem}
	\label{only_cyclic_amenable_subgroups_plump}
		Assume $G$ is such that any amenable subgroup is cyclic, $Z(G) = \set{e}$, and in addition, no two elements have finite coprime order. Then any nontrivial normal subgroup of $G$ is relatively C*-simple.
	\end{theorem}

	\begin{remark}
		It is worth noting that this last requirement that $G$ should have no two elements of finite coprime order is true for a large class of groups, including torsion-free groups and $p$-groups.
	\end{remark}	
	
	Whether or not $G$ has trivial center is surprisingly sufficient in determining whether $G$ is C*-simple or not. In the case of countable groups, \cite[Theorem~6.12]{breuillard_kalantar_kennedy_ozawa_c_simplicity} tells us it suffices to prove that $R_a(G) \subseteq Z(G)$. The argument we present here avoids countability, but requires a bit of a detour into theory on the Furstenberg boundary.
	
	\begin{lemma}
	\label{lemma_furstenberg_stabilizers_intersection_conjugacy_class}
		Let $G$ denote any discrete group, and let $x \in \boundary_F G$ be arbitrary. Letting $G_x$ denote the point-stabilizer of $x$, if $s \in G_x$ is nontrivial, and $y_1, \dots, y_n \in \boundary_F G$, then $G_{y_1} \cap \dots \cap G_{y_n}$ always contains some conjugate of $s$.
	\end{lemma}

	\begin{proof}
		This is a special case of \cite[Lemma~3.7]{breuillard_kalantar_kennedy_ozawa_c_simplicity} obtained by setting $U = \Fix(s)$ (necessarily clopen - see \Cref{frolik_theorem_furstenberg_boundary}), $\varepsilon = \frac{1}{n}$, and $\mu = \frac{1}{n} (\delta_{y_1} + \dots + \delta_{y_n})$. Our end result is that there is some $r \in G$ with $ry_i \in \Fix(s)$ for all $i$, or in other words, $r^{-1}sr \in G_{y_i}$ for all $i$.
	\end{proof}
	
	\begin{proposition}
	\label{only_cyclic_amenable_subgroups_c_simple}
		Assume $G$ has the property that any amenable subgroup is cyclic. Then $G$ is C*-simple if and only if $Z(G) = \set{e}$.
	\end{proposition}

	\begin{proof}
		As the center is always an amenable normal subgroup, any C*-simple group $G$ must have trivial center. Conversely, assume $G$ has trivial center. We will first show that $G$ has trivial amenable radical. Given any $t \in G$, we have that $(\left<t\right> R_a(G)) / R_a(G) \isoto \left<t\right> /(\left<t\right> \cap R_a(G))$, which is amenable, and so by extension, $\left<t\right> R_a(G)$ is amenable, thus cyclic. This shows $t$ commutes with every element of $R_a(G)$. Since $t$ was arbitrary, $R_a(G) \subseteq Z(G) = \set{e}$.
		
		Now we wish to show that none of the point-stabilizers $G_x$ for $x \in \boundary_F G$ can be nontrivial. Assume otherwise. Recall that $G_x$ is always amenable - see, for example \cite[Proposition~2.7]{breuillard_kalantar_kennedy_ozawa_c_simplicity}. This tells us that $G_x = \left<s\right>$ for some $x \in \boundary_F G$ and $s \neq e$. If $G_x$ were finite, it follow from $\bigcap_{y \in \boundary_F G} G_y = R_a(G) = \set{e}$ that there are $y_1, \dots, y_n \in \boundary_F G$ with $G_x \cap G_{y_1} \cap \dots \cap G_{y_n} = \set{e}$. This contradicts \Cref{lemma_furstenberg_stabilizers_intersection_conjugacy_class}. If $G_x$ is infinite cyclic, this tells us that there is some $y \in \boundary_F G$ with $G_x \neq G_y$. Without loss of generality, $G_x \not \subseteq G_y$, and so $G_x \cap G_y = \left<s^n\right>$ for some $\card{n} \geq 2$. Again, \Cref{lemma_furstenberg_stabilizers_intersection_conjugacy_class} gives us that there is some $r \in G$ with $rsr^{-1} = s^m$ for some $\card{m} \geq 2$. It is easy to show that, inductively, $r^k s r^{-k} = s^{m^k}$, and so $r^k \left<s\right> r^{-k}$ converges to $\set{e}$ in the Chabauty topology. This can never happen if $G_x \neq \set{e}$, as $\set{G_x}_{x \in \boundary_F G}$ is an amenable uniformly recurrent subgroup - see \cite[Remark~4.3]{kennedy_c_simplicity_intrinsic}.
	\end{proof}

	\begin{proof}[Proof of \Cref{only_cyclic_amenable_subgroups_plump}]
		By \Cref{only_cyclic_amenable_subgroups_c_simple}, any such group is C*-simple. Assume $N$ is a nontrivial normal subgroup, and $C_G(N) \neq \set{e}$. We know that $G$ being C*-simple implies $N$ is C*-simple by \cite[Theorem~1.4]{breuillard_kalantar_kennedy_ozawa_c_simplicity}, and so $Z(N) = N \cap C_G(N)$ is trivial. Thus, if we choose nontrivial $x \in N$ and $y \in C_G(N)$, then $\left<x,y\right> \isoto \left<x\right> \times \left<y\right>$. Such a group is both amenable and non-cyclic, contradicting our assumption. Hence, any nontrivial normal subgroup has trivial centralizer, and so by \Cref{normal_plump_subgroup_amrutam_equivalence}, we are done.
	\end{proof}
	
	Recall that the free Burnside group $B(m,n)$ is the universal group on $m$ generators satisfying $x^n = e$ for all elements $x$ in the group. The Burnside problem, which was one of the largest open problems in group theory, asked whether such groups are always finite. The answer, as it turns out, is no, and in addition, some of these groups are C*-simple - see \cite[Corollary~6.14]{breuillard_kalantar_kennedy_ozawa_c_simplicity}. In particular, they remark that for any $m \geq 2$ and $n$ odd and sufficiently large, any non-cyclic subgroup contains a copy of the non-amenable group $B(2,n)$. Hence, we obtain the following:
	
	\begin{example}
		Assume $m \geq 2$ and $n$ is prime and sufficiently large. Then any non-trivial normal subgroup of $B(m,n)$ is relatively C*-simple.
	\end{example}

	\subsection{A remark on Thompson's group $F$}
	
	Thompson's group $F$ was the original candidate counterexample for the now-disproven Day-von Neumann conjecture, which stated that a group is non-amenable if and only if it contains a copy of $\F_2$, the free group on two generators. A good introduction to this, and related groups, can be found in \cite{cannon_floyd_parry_thompsons_groups}. It is known that $F$ does not contain a copy of $\F_2$, but whether or not it is amenable is still a large open question in group theory. However, it is known that $F$ is non-amenable if and only if it is C*-simple - see \cite[Theorem~1.6]{le_boudec_matte_bon_thompsons_groups}. Hence, with a bit of extra work, we obtain the following equivalence:
	
	\begin{theorem}
		Thompson's group $F$ is non-amenable if and only if its derived subgroup $[F,F]$ is relatively C*-simple in $F$.		
	\end{theorem}

	\begin{proof}
		Relative simplicity of $[F,F]$ in $F$ would imply that both $[F,F]$ and $F$ are C*-simple, in particular non-amenable. It remains to prove the converse.
		
		Assume $F$ is non-amenable, hence C*-simple. It is known that every proper quotient of $F$ is abelian \cite[Theorem~4.3]{cannon_floyd_parry_thompsons_groups}, or equivalently, $[F,F] \subseteq N$ for any normal subgroup $N \normal F$ with $N \neq \set{e}$. In particular, we must have that $C_F([F,F]) = \set{e}$, otherwise $[F,F]$ would be abelian (and thus $F$ would be amenable). By \Cref{normal_plump_subgroup_amrutam_equivalence}, we are done.
	\end{proof}

	This shows, for example, that to prove amenability of $F$, it would suffice to construct a non-faithful unital completely positive map $\phi : C^*_\lambda(F) \to B(\mathcal{H})$ that is a *-homomorphism on $C^*_\lambda([F,F])$.

	\bibliographystyle{amsalpha}
	\bibliography{relatively_c_simple_subgroups}

\providecommand{\bysame}{\leavevmode\hbox to3em{\hrulefill}\thinspace}
\providecommand{\MR}{\relax\ifhmode\unskip\space\fi MR }
\providecommand{\MRhref}[2]{%
  \href{http://www.ams.org/mathscinet-getitem?mr=#1}{#2}
}
\providecommand{\href}[2]{#2}
\begin{thebibliography}{BKKO17}

\bibitem[{Amr}18]{amrutam_plump_subgroups}
Tattwamasi {Amrutam}, \emph{{On Intermediate C*-subalgebras of C*-simple Group
  Actions}}, arXiv e-prints (2018), arXiv:1811.11381.

\bibitem[BKKO17]{breuillard_kalantar_kennedy_ozawa_c_simplicity}
Emmanuel Breuillard, Mehrdad Kalantar, Matthew Kennedy, and Narutaka Ozawa,
  \emph{{C*-simplicity and the unique trace property for discrete groups}},
  Publications math{\'e}matiques de l'IH{\'E}S \textbf{126} (2017), no.~1,
  35--71.

\bibitem[CFP96]{cannon_floyd_parry_thompsons_groups}
James~W. Cannon, William~J. Floyd, and Walter~R. Parry, \emph{{Introductory
  notes on Richard Thompson's groups}}, Enseignement Math{\'e}matique
  \textbf{42} (1996), 215--256.

\bibitem[Fro71]{frolik_extremally_disconnected_fixed_points}
Zden{\v{e}}k Frol{\'i}k, \emph{Maps of extremally disconnected spaces, theory
  of types, and applications}, General Topology and Its Relations to Modern
  Analysis and Algebra, Proceedings of the Kanpur topological conference, 1968
  (Stanley~P. Franklin, Zden{\v{e}}k Frok{\'i}k, and V{\'a}clav Koutn{\'i}k,
  eds.), Academia Publishing House of the Czechoslovak Academy of Sciences,
  1971, pp.~131--142.

\bibitem[Fur73]{furstenberg_boundary_theory}
Harry Furstenberg, \emph{Boundary theory and stochastic processes on
  homogeneous spaces}, Harmonic Analysis on Homogeneous Spaces (Calvin~C.
  Moore, ed.), Proceedings of Symposia in Pure Mathematics, vol.~26, American
  Mathematical Society, Providence, Rhode Island, 1973, pp.~193--229.

\bibitem[Gla76]{glasner_proximal_flows}
Shmuel Glasner, \emph{Proximal flows}, Lecture Notes in Mathematics, vol. 517,
  Springer-Verlag, Berlin Heidelberg, 1976.

\bibitem[Haa16]{haagerup_new_look_c_simplicity}
Uffe Haagerup, \emph{A new look at {C}*-simplicity and the unique trace
  property of a group}, Operator Algebras and Applications (Cham) (Toke~M.
  Carlsen, Nadia~S. Larsen, Sergey Neshveyev, and Christian Skau, eds.),
  Springer International Publishing, 2016, pp.~167--176.

\bibitem[{Ken}15]{kennedy_c_simplicity_intrinsic}
Matthew {Kennedy}, \emph{{An intrinsic characterization of C*-simplicity}},
  arXiv e-prints (2015), arXiv:1509.01870, to appear in Annales Scientifiques
  de l'{\'E}cole Normale Sup{\'e}rieure.

\bibitem[KK17]{kalantar_kennedy_boundaries}
Mehrdad Kalantar and Matthew Kennedy, \emph{{Boundaries of reduced C*-algebras
  of discrete groups}}, Journal f{\"u}r die reine und angewandte Mathematik
  (Crelles Journal) \textbf{2017} (2017), no.~727, 247--267.

\bibitem[LM18]{le_boudec_matte_bon_thompsons_groups}
Adrien {Le Boudec} and Nicol{\'a}s {Matte Bon}, \emph{{Subgroup dynamics and
  C*-simplicity of groups of homeomorphisms}}, Annales Scientifiques de
  l'{\'E}cole Normale Sup{\'e}rieure \textbf{51} (2018), no.~3, 557--602.

\bibitem[{Mon}19]{monod_relative_furstenberg_boundary}
Nicolas {Monod}, \emph{{Furstenberg boundaries for pairs of groups}}, arXiv
  e-prints (2019), arXiv:1902.08513.

\bibitem[Oza14]{ozawa_c_simplicity}
Narutaka Ozawa, \emph{{Lecture on the Furstenberg boundary and C*-simplicity}},
  Available at
  \url{http://www.kurims.kyoto-u.ac.jp/~narutaka/notes/yokou2014.pdf}, 2014.

\bibitem[PS79]{paschke_salinas_free_products_c_simple}
William Paschke and Norberto Salinas, \emph{{C*-algebras associated with free
  products of groups}}, Pacific Journal of Mathematics \textbf{82} (1979),
  no.~1, 211--221.

\end{thebibliography}
\end{document}